\documentclass[12pt]{amsart}
\usepackage{amsfonts}

\usepackage{graphicx}
\usepackage{color}
\theoremstyle{plain}
\newtheorem{thm}{Theorem}[section]
\newtheorem{lem}[thm]{Lemma}
\newtheorem{cor}[thm]{Corollary}

\newtheorem{que}[thm]{Question}

\theoremstyle{definition}
\newtheorem{defn}[thm]{Definition}
\newtheorem{const}[thm]{Construction}

\newtheorem{remark}[thm]{Remark}
\begin{document}

\title[Simply Connected  3-manifolds with Rigid Genus One Ends]
{Simply Connected Open 3-manifolds with \\Rigid Genus One Ends}
\author{Dennis Garity}
\address{Mathematics Department, Oregon State University,
Corvallis, OR 97331, U.S.A.}
\email{garity@math.oregonstate.edu}
\urladdr{http://www.math.oregonstate.edu/\symbol{126}garity}

\author{Du\v{s}an Repov\v{s}}
\address{Faculty of Education,
and Faculty of Mathematics and Physics,
University of Ljubljana,  P.O.Box 2964,
Ljubljana, Slovenia}
\email{dusan.repovs@guest.arnes.si}
\urladdr{http://www.fmf.uni-lj.si/\symbol{126}repovs}

\author{David Wright}
\address{Brigham Young University,
Provo, UT 84602, U.S.A.}
\email{wright@math.byu.edu}
\urladdr{http://www.math.byu.edu/\symbol{126}wright}

\date{\today}



\begin{abstract} 
We construct uncountably many  simply connected open 3-manifolds with genus one ends 
homeomorphic to the Cantor set. Each constructed manifold has the property  that any self 
homeomorphism of the manifold (which necessarily extends to  a homeomorphism of the ends) 
fixes the ends pointwise.
These manifolds are complements of rigid generalized Bing-Whitehead (BW) Cantor sets. Previous 
examples of rigid Cantor sets with simply connected complement in $R^{3}$ had infinite genus and 
it was an open question as to whether finite genus examples existed. The examples here  exhibit 
the minimum possible genus, genus one. These rigid generalized BW Cantor sets are 
constructed using variable numbers of Bing and Whitehead links. Our previous result with \v{Z}eljko 
determining when BW Cantor sets are equivalently embedded in $R^{3}$ extends to 
the generalized construction. This characterization is used to prove rigidity and to  distinguish the  
uncountably many examples.
\end{abstract}
\maketitle


\section{Introduction}
\label{introsection}

Each Cantor set $C$ in $S^{3}$ has complement an open $3$-manifold with end set $C$. Properties of the embedding of the Cantor set give rise to properties of the corresponding complementary $3$-manifold. See \cite{SoSt12} for an example of this and \cite{Beco87} for another setting in which wild Cantor sets in $S^{3}$ arise.

In \cite{GaReZe06} new examples  were constructed of Cantor sets that were both rigidly embedded and had simply connected complement. Prior to that, only examples having one  of these properties were known. The Antoine type \cite{An20} rigid Cantor sets \cite{Shi74} failed to have simply connected complement. Their rigidity was a consequence of Sher's result \cite{Sh68} that if two Antoine Cantor sets are equivalently embedded, then the stages in the defining sequences must match up exactly. 

The  new examples \cite{GaReZe06} had points in the Cantor set of arbitrarily large local genus and thus the complementary $3$-manifold did not have finite genus at infinity. The rigidity here was a consequence of a dense set of points in the Cantor set having distinct local genera. The simple connectivity of the complement was a consequence of a Skora type \cite{Sk86} construction for the stages. The question arose as to whether finite genus examples possessing both properties were possible. For any such examples, rigidity would have to be determined in some other manner since local genus could no longer be used.

In this paper, we construct examples of Cantor sets that are rigid, have simply connected complement,  and have local (and global) genus one (See Theorem \ref{main theorem}). These examples exhibit the minimal possible genus. For the corresponding statement concerning simply connected open $3$-manifolds of genus one at infinity with rigid end structure, see Theorem \ref{maintheorem3}. 

In Section \ref{backgroundsection}, we give definitions and the basic results needed for working with generalized Bing-Whitehead (BW) Cantor sets. In the following section, Section \ref{generalizedsection}, we generalize the results from our previous paper with \v{Z}eljko \cite{GaReWrZe11} that are needed in the new setting. In Section \ref{mainresultssection} we state and prove the main results about Cantor sets. In Section \ref{3manifoldsection} we state the corresponding results for the complementary $3$-manifolds. Section \ref{question section} lists some remaining questions.

\section{Preliminaries}
\label{backgroundsection}

\subsection{Background information}
Refer to \cite{DeOs74}, \cite{Wr89} and \cite{GaReWrZe11} for results about BW Cantor sets, to 
\cite{GaReZe05, GaReZe06} for a discussion of Cantor sets in general and rigid Cantor sets, and to \cite{Ze05} for results about local genus of points in Cantor sets and defining sequences for Cantor sets. See \cite{CaMeZa12} for results about dynamics of self-maps of the Cantor set. The bibliographies in these papers contain additional references to results about Cantor sets.
For background on Freudenthal compactifications and theory of ends, see \cite{Fr42}, \cite{Di68}, and \cite{Si65}.

Two Cantor sets $X$ and $Y$ in $S^3$ are  \emph{equivalent} if there is
a self homeomorphism of $S^3$  taking $X$ to $Y$. If there is no such homeomorphism, the Cantor sets are said to be \emph{inequivalent}, or \emph{inequivalently embedded}. A Cantor set $C$ is \emph{rigidly embedded} in $S^{3}$ if the only self homeomorphism of $C$ that extends to a homeomorphism of $S^{3}$ is the identity. This is in marked contrast to the standard embedding of a Cantor set in $S^{3}$ for which every self homeomorphism extends. {See Daverman \cite{Da79} for nonstandard embeddings of Cantor sets with this strong homeomorphism extension property.}

 Let $T$ be a solid torus. Throughout this paper, we
assume that the tori we are working with are \emph{unknotted}
in  $S^{3}$. (Our results and constructions also work in $R^{3}$.) 
A \emph{Bing link} in $T$ is a union of 2
linked solid tori $F_1\cup F_2$ embedded in $T$ as shown in
Figure \ref{BingWhitehead}. A \emph{Whitehead link} in $T$ is a
solid torus $W$ embedded in $T$ as shown in Figure \ref{BingWhitehead}. 
For background details and terminology, see Wright's paper
\cite{Wr89}. 

\begin{figure}[htb]
\begin{center}
\includegraphics[width=.85\textwidth]{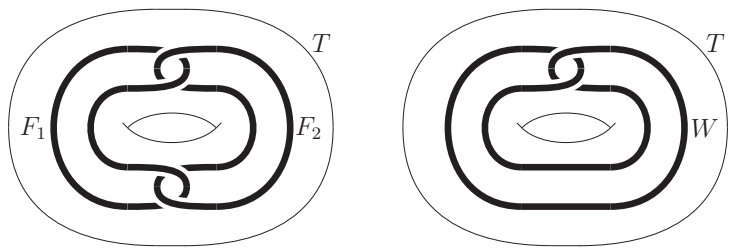}
\end{center}
\caption{Bing and Whitehead links}
\label{BingWhitehead}
\end{figure}

\subsection{Construction of (generalized) Bing-Whitehead compacta}

For completeness and consistency of notation, we outline the steps in the construction of standard Bing-Whitehead  (BW) compacta. 
Let $M_{0}$ be an unknotted torus in $S^{3}$, and $M_{1}$ be obtained from $M_{0}$ by placing a Bing link in $M_{0}$. 
Inductively obtain $M_{k}$ from $M_{k-1}$ \emph{either} by placing a Bing link in each component of $M_{k-1}$ \emph{or} by placing a 
Whitehead link in each component of $M_{k-1}$. Let $m_1$ denote the number of consecutive Whitehead links placed in $M_1$ before the second Bing link occurs, and let $m_k$ denote the number of consecutive Whitehead links that occur between the $k$-th and $(k+1)$st Bing links. 
\begin{defn}
\label{BWcompactum}
The \emph{standard Bing-Whitehead compactum} associated with this construction is defined to be 
$$X=\bigcap_{i=0}^{\infty}M_{i}\text{ \ \   and is denoted }X=BW(m_1,m_2,\ldots).$$
\end{defn}

We also define $M_{i}, i<0$, so that $M_{i}$ is a Whitehead link in $M_{i-1}$ and let $X_{M}^{\infty}$ be $\bigcap_i\left(S^{3}\setminus M_{i}\right)$. 
The set $X_{M}^{\infty}$ is called \emph{the compactum at infinity} associated with $X$ corresponding to the defining sequence $M=(M_{i})$.

This approach leads to the association of a BW compactum with a BW (infinite) labeled binary tree where the vertices correspond to one of the two torus components of a Bing link placed in a torus and the labels on the vertices correspond to the number of Whitehead constructions to perform in the torus associated with that vertex before performing another Bing construction. See Figure \ref{genBW} where $m_{1}=w_{1,1}=w_{1,2}$, $m_{2}=w_{2,1}= \cdots = w_{2,4}$, and so on.

Note that in this standard BW construction we require that all tori in $M_{k}$ are obtained \emph{by the same}
(Bing or Whitehead) \emph{construction} from the respective tori in $M_{k-1}$. We also assume that infinitely many of the $M_{i}, i>0,$ arise from Bing constructions and that infinitely many of them arise from Whitehead constructions.

\begin{figure}[htb]
\begin{center}
\includegraphics[width=.70\textwidth]{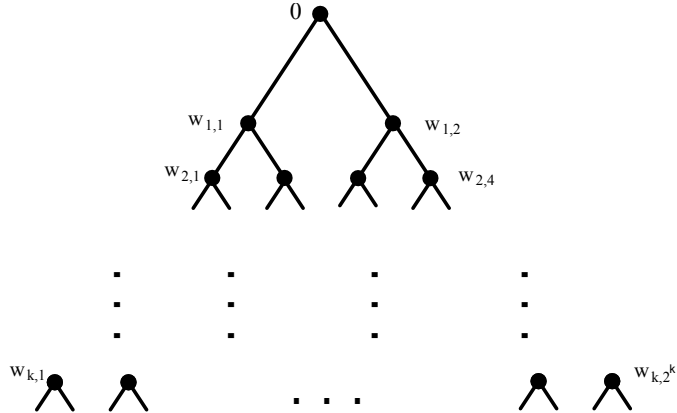}
\end{center}
\caption{  Bing-Whitehead tree}
\label{BWTree2}
\end{figure}

In a generalized BW construction we require that all tori in $M_{k}$ are either obtained by 
Bing or Whitehead construction from the respective tori in $M_{k-1}$, \emph{but not necessarily all of the same kind}.

{The construction can also be associated with a labeled binary tree (see Figure \ref{BWTree2}).
The top node represents an unknotted solid torus in $S^{3}$ denoted by $M_{0}$ or $T_{0,1}$. }Each of the successive nodes at depth $i$ represents an unknotted torus $T_{i,j}, j=1\ldots 2^i$, which is one component of a Bing link in the immediately  preceding stage. 
The numbers $w_{i,j}\in\{0,1,2,\ldots\}$ on nodes denote the number of consecutive Whitehead constructions placed in the torus associated with that node before the next Bing construction occurs.
Define $M_i=\bigcup_{j=1}^{2^i} T_{i,j}$.

\begin{defn}
\label{genBW}The \emph{generalized BW compactum } associated with this construction is defined to be 
$$
X=\bigcap_{i=0}^{\infty}M_{i}\text{\ \  and is denoted }X=BW(z_{1},z_{2},\ldots)
$$ 
where
$z_{i}=[w_{i,1},w_{i,2},\ldots,w_{i,2^{i}}]$.
\end{defn}

As in the standard construction, we also define $M_{i}, i<0$, so that $M_{i}$ is a Whitehead link in $M_{i-1}$,  and let $X^{\infty}$ be $\bigcap_i\left(S^{3}\setminus M_{i}\right)$.

\medskip

Consider now a  component $x$ in a generalized BW compactum $X$. This component corresponds to a unique {ray} starting  at the root vertex in the associated binary tree. The vertex labels along this {ray} form a sequence $s(x)=(0, w_{1,x(1)},w_{2,x(2)}\ldots)$. 

\begin{defn} \label{Wsequence}The sequence $s(x)=(0, w_{1,x(1)},w_{2,x(2)}, \ldots)$ is called the \emph{Whitehead sequence associated with the component} $x$ with respect to the generalized BW compactum  $X = BW(z_1,z_2,\ldots)$.
\end{defn}
This will be of interest in the next few sections, particularly when the components consist of single points so that the compactum $X$ is a Cantor set.

\subsection{Geometric Index}
If $K$ is a link in the interior of a solid torus $T$, then we denote the \emph{geometric index} of $K$ in $T$ by $N(K,T)$.  This is the minimum  of the cardinality $|K \cap D|$ over all meridional disks $D$ of $T$.    If $T$ is a solid torus and $M$ is a finite union of disjoint solid tori so that $M \subset \text{Int}(T_1)$, then the geometric index $N( M,T)$ of $M$ in  $T$ is $N(K,T)$ where $K$ is a core of $M$.  The geometric index of a Bing link in a torus $T$ is 2.  The geometric index of a Whitehead link  in a torus $T$ is also 2. Let $T_0$ and $T_1$ be unknotted solid tori in $S^{3}$ with  $T_0 \subset \text{Int}(T_1)$.  Then $\partial T_0$ and  $\partial T_1$ are parallel if the manifold $T_1 -  \text{Int} (T_0)$ is homeomorphic to $\partial T_0 \times I$ where $I$ is the closed unit interval $[0,1]$.

\begin{remark}\label{geometric index}
See Schubert \cite{Sc53} and \cite{GaReWrZe11} for the following needed results about geometric index.
\begin{itemize}
\item Let $T_0$ and $T_1$ be unknotted solid tori in $S^{3}$ with  $T_0 \subset\text{Int}(T_1)$ and $N( T_0, T_1) = 1$.  Then $\partial T_0$ and  $\partial T_1$ are parallel.
\item Let $T_0$ be a finite union of disjoint tori. Let $T_1$ and $T_2$ be tori so that $T_0 \subset \text{int}T_1$ and $T_1 \subset \text{Int}T_2$.  Then $N(T_0, T_2) =  N(T_0, T_1) \cdot  N(T_1, T_2)$.
\item Let $T$ be a torus in $S^{3}$ and let $T_{1},T_{2}$ be unknotted tori in $T$, each of geometric index $0$ in $T$. Then the geometric index of $\cup_{i=1}^{2}T_{i}$ in $T$ is even.
\end{itemize}
\end{remark}

\section{Results on generalized Bing-Whitehead compacta}
\label{generalizedsection}

\subsection{Constructions yielding Cantor sets}\ \\ 
We are primarily interested in generalized BW compacta that are Cantor sets. At one extreme, if no Whitehead links are used, then the construction can result in a Bing Cantor set \cite{Bi52}. At the other extreme, if no Bing links are used, the resulting compactum has a single component which is the Whitehead continuum and the decomposition of $S^{3}$ that shrinks this to a point does not yield $S^{3}$.

For standard BW compacta, Ancel and Starbird \cite{AnSt89}, and independently {Wright} \cite{Wr89}, gave a complete characterization of when the construction can be done so as to yield a Cantor set in $S^{3}$. What is needed is sufficiently many Bing links, or equivalently sufficiently few Whitehead links in the process. For a generalized BW construction,  
let $w_{i} = \max(z_{i}) = \max [w_{i,1},w_{i,2},\ldots,w_{i,2^{i}}]$.  An examination of Appendix A in \cite{Wr89} shows that the argument there proves the following result:

\begin{thm}\label{Whiteheadlimit}
 {If \ $\sum_{i=1}^{\infty}\dfrac{1}{2^{\sigma_{i}}} =\infty$ where $\sigma_{i}=w_1 + w_2 + \cdots + w_i$, then the decomposition 
  given
  by points and components of the generalized BW compactum $X$ is shrinkable.}
\end{thm}

This result will guide us in the next section in the construction associated with  Theorem 
\ref{main theorem}.

\subsection{Properties of BW Constructions and Compacta}\noindent

Many of the results in \cite{Wr89} and \cite{GaReWrZe11} about standard BW constructions carry over directly with the same proofs to the setting of generalized BW constructions. We list some of the results from the 1989 paper, restated in the general setting, that will be useful for our main results.

\begin{lem}\label{BW properties}
Let $M$ be a Bing or a Whitehead link in  a torus $T$. Let X be a generalized BW compactum with defining sequence $M=(M_{i})$ and $X_{M}^{\infty}$ the associated continuum at infinity.
Then the following holds:
\begin{itemize}
\item \cite[Theorem 4.6 ]{Wr89} No sphere in the complement of $X\cup X^{\infty}$ separates $X\cup X^{\infty}$.
\item
\cite[Theorem 4.3 ]{Wr89} A loop on the boundary of $M_{i}$ 
is essential in the boundary of $M_{i}$
if and only if it is essential in the complement of $X \cup X^{\infty}$.
\item \cite[Theorem 4.4]{Wr89}
If loops $\ell_{1}\text{ and }\ell_{2}$ in $\partial M_{i}$ and 
$\partial M_{j}$ respectively, $i\neq j$, are homotopic in the complement of 
$X\cup X^{\infty}$, then they are inessential in $X\cup X^{\infty}$.
\item \cite{DeOs74}, \cite[Appendix B]{Wr89}{\rm{)}}
The complement of a generalized BW compactum in $S^{3}$ is simply connected.
\end{itemize}
\end{lem}

\subsection{Matching up Stages} The key to proving rigidity for certain generalized BW Cantor sets is to show that if such a Cantor set $C$ has two defining sequences $(M_{i})$ and $(N_{i})$ then the components can be made to match up. For completeness, we provide  outlines of the proofs of key results: Lemma \ref{remove curves}, Lemma \ref{match stage}, Corollary \ref{matchWsequences} and Theorem \ref{sametails}. For more complete details, see  \cite{GaReWrZe11}.

\begin{lem}\label{remove curves} Let $X$ be a generalized BW compactum  
with two defining BW  sequences $(M_{i})$  and  $(N_{j} )$. Let $T$ be a component of
some $M_{i}$ and let $T^{\prime}$ be the component of $M_{i-1}$ containing $T$.
If $T$ lies in the interior of a component $S^{\prime}$ of some $N_{j}$, and $S^{\prime}\subset M_{0}$, then there is a homeomorphism $h$ of $S^{3}$, fixed on $X\cup (S^{3}-(T^{\prime}\cup S^{\prime}))$, so that $h( \partial (T))\cap \partial (N_{j+1}\cap S^{\prime})=\emptyset$.
\end{lem}
\begin{proof} After a general position adjustment, we may assume that $ \partial (T)\cap \partial (N_{j+1}\cap S^{\prime})$ consists of a finite collection of simple closed curves. Since 
a curve on the boundary of $T$ is essential if and only if it is essential in the complement of $X \cup X_{M}^{\infty}$ and a curve in the boundary of $N_{j+1}$ is essential if and only if it is essential in the complement of $X\cup X_{N}^{\infty}$, we see that a curve on $\partial T$ is trivial on $ \partial (T)\cap N_{j+1}$ is trivial on $\partial T$ if and only if it is trivial on $\partial N_{j+1}$.

\textbf{Removing Trivial Curves of Intersection:} If there are any trivial curves, choose a component $S$ of $S^{\prime}\cap N_{j+1}$  that contains such a curve in $\partial S$.  Choose an innermost trivial simple closed curve $\alpha$ on $\partial S$.  Since $\alpha$ is innermost, it bounds a disk $D^{\prime}$ with interior missing  $\partial T$.  The curve $\alpha$ also bounds a disk $D$ in $\partial T$.

The $2$-sphere $D\cup D^{\prime}$ bounds a three-cell  in $S^{\prime}\cap T^{\prime}$
that  contains no points of $X$. Use this three-cell to push $D$ onto $D^{\prime}$ and then a little past $D^{\prime}$ into a collar on the cell by a homeomorphism $h$ of $S^{3}$. This homeomorphism can be chosen to fix $X$, $S^{3}- S^{\prime}$, and $S^{3}-T^{\prime}$. 
This has the result that  $h(\partial T)\cap \partial (S^{\prime}\cap N_{j+1})$ has fewer curves of intersection than $\partial T \cap \partial (S^{\prime}\cap N_{j+1})$, and so that no new curves of intersection with $\partial (N_{j+1})$ are introduced. Continuing this process eventually removes all trivial curves of intersection of $\partial T \cap \partial (S^{\prime}\cap N_{j+1})$.

\textbf{Removing Nontrivial Curves of Intersection:} At this point,  there is at most one component $S$ of $N_{j+1}\cap S^{\prime}$ for which  $\partial T \cap \partial S\neq\emptyset$.  These remaining curves of intersection on $\partial T$ must be parallel $(p,q)$ torus curves and the corresponding curves on $\partial S$ must be parallel $(s,t)$ curves.
If both $p$ and $q$ are greater than 1, so that the torus curve is a nontrivial knot, then $(s,t)=(p,q)$ or $(s,t)=(q,p)$.

To simplify notation in what remains, we will refer to $h(T)$ as (the new) $T$.  
We  work towards removing these remaining curves of intersection of the boundaries, so that  $T$ is either contained in a component $N_{j+1}\cap S$  or contains the components of $(N_{j+1}\cap S)$.  Consider an annulus $A$ on the boundary of $S$ bounded by two adjacent curves of the intersection of $\partial S$ and $\partial T$. Choose this annulus so that its interior lies in the interior of $T$. We consider separately all possibilities for how the boundary curves of $A$ lie on $S$. 

{\setlength{\parindent}{0pt}\emph{I. Curves of intersection on $T$ that are $(p,q)$ curves for $p\geq 2$.}}

Consider a meridional disk $D$ for $T$ in general position with respect to $A$ so that $D\cap A$ consists of $p$ arcs intersecting the boundary of $D$ in endpoints and of simple closed curves. Figure \ref{meridian remove} illustrates a possible situation when $p=5 \text{ and } q=3$ and trivial curves of intersection have been removed. The solid unlabeled disks indicate the intersection of the next stage  $M_{i+1}$ with $D$.

The intersection of $A$ with $D$ can be adjusted using cut and paste techniques similar to that used in the trivial curve case so that the end result is intersections as in one of the two cases in  Figure  \ref{meridian remove}.

\begin{figure}[htb]
\begin{center}
\includegraphics[width=.4\textwidth]{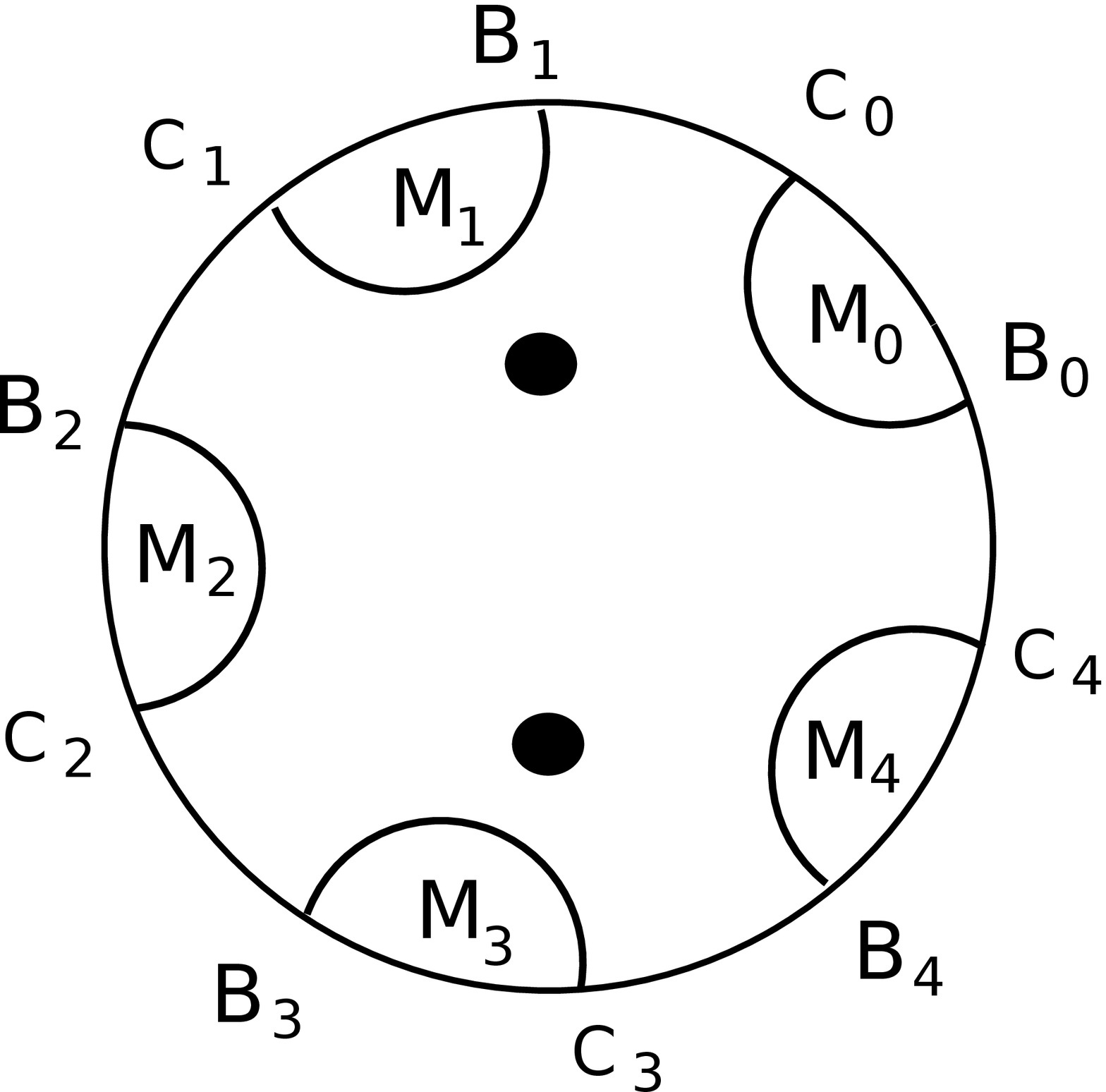}
\ \ \includegraphics[width=.4\textwidth]{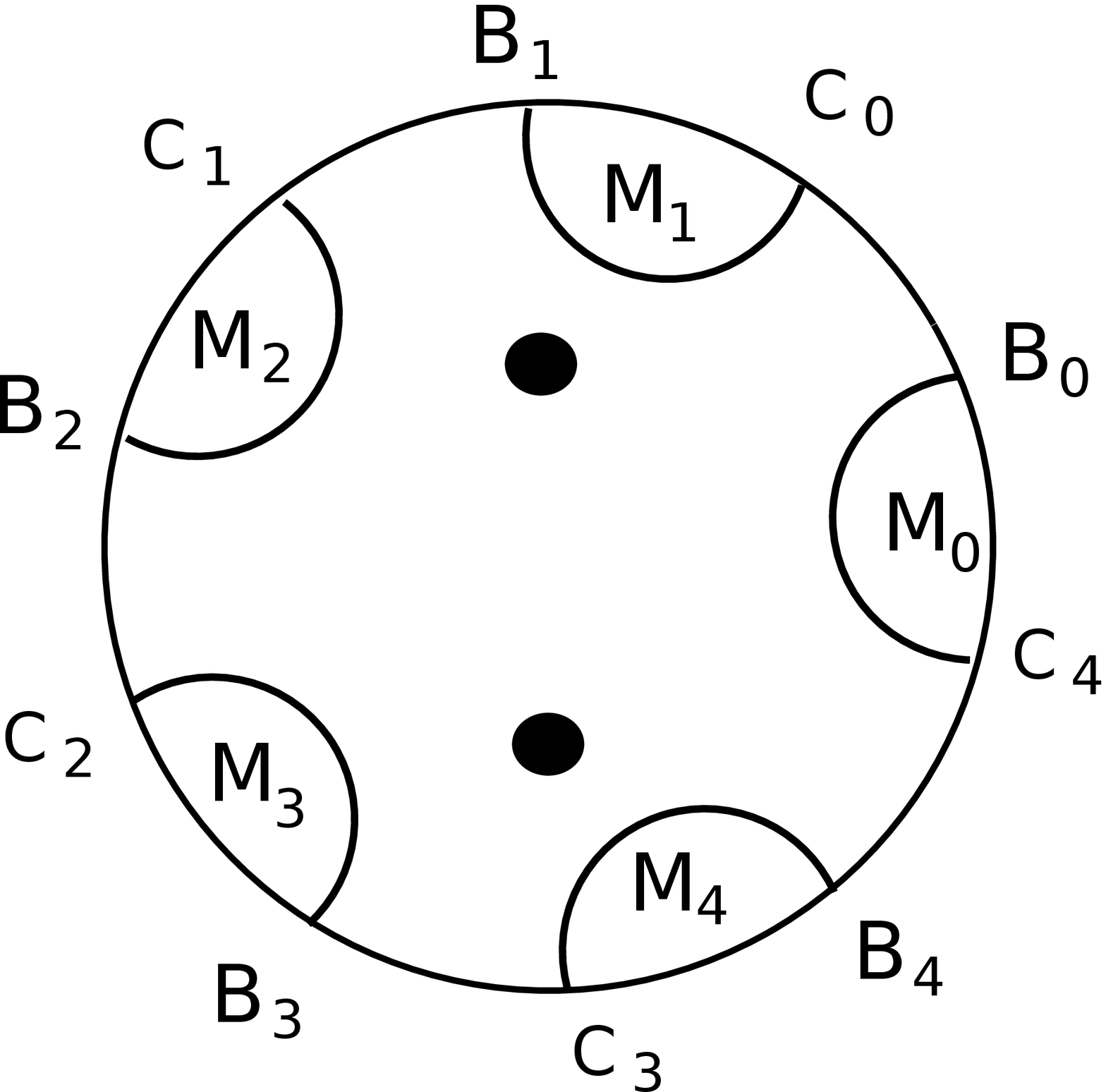}
\end{center}
\caption{Meridional Disk $D$ of $T$ after Adjustment}
\label{meridian remove}
\end{figure} 

Each of the regions labeled $M_{i}$ can be shown to be a meridional disk of a torus $T_{1}$ that is contained in $T$. This torus will then be used to push across and remove the intersections of $A$ with $T$. 
The intersection of $T$ with $S$ corresponding to $A$ can now be removed by a homeomorphism of $S^{3}$ fixed on $X$ and on the complement of a small neighborhood of $T$ that takes $A$ through $T_{1}$ to an annulus parallel to $A_{1}$ and just outside of $T$. Inductively, all curves of intersection of $T$ with $S$ can be removed by a homeomorphism of $S^{3}$ fixed on $X$ and the complement of a small neighborhood of $T$.

{\setlength{\parindent}{0pt}\emph{II. Curves of intersection on $S^{\prime}$ that are $(p,q)$ curves for 
$p=1$.}}

An argument similar to that above can be used. 

{\setlength{\parindent}{0pt}\emph{III. Curves of intersection on $S^{\prime}$ that are $(p,q)$ curves for 
$p=0$.}}

In this case the curve is a $(0,q)$ curve for the torus $T$, but it is a $(q,0)$ curve for the complementary torus with $q \ne 0$.  In this case there is an annulus $A$ on the boundary of $T$ that has its interior in the exterior of $T$, so that the intersection of $A$ with the boundary of  $T$ consists of curves in the intersection of the boundaries of $T$ and  $S^{\prime}$.  This essentially turns the problem inside out, and we can use the previous methods to push $A$ to the interior of  $T$ fixed on a slightly shrunken  $T$, all the other components of $M_i$, and the complement of $M_{i-1}$.

\end{proof}

\begin{lem} \label{match stage}Assume that $X$ is a generalized BW compactum  
with two defining BW  sequences $(M_{i})$  and  $(N_{j} )$. Choose $N$ so that
the components of $M_{N}$ are in the interior of $N_{0}$. Let $T$ be a component of
$M_{N}$. Then there is an integer $K$ and a homeomorphism $h$ of $S^{3}$, fixed on $X \cup (S^{3}-(N_{0}\cup M_{N-1}))$ so that $h(T)$ is interior to components of $N_{i}, i < K$ and so that $h(T)$ is a component of $N_{K}$.
\end{lem}
\begin{proof} Choose an integer $L$ so that the components of $N_{L}$ are interior to $M_{N}$.
Inductively apply Lemma \ref{remove curves} producing successive homeomorphisms so that the image of $\partial(T)$ is disjoint from $\partial(N_{i})$ until a first stage $K\leq L$ is reached where a component of  $N_{K}$ is interior to the image of $T$.

At this point, there is a homeomorphism $h$ of $S^{3}$, fixed on $X\cup (S^{3}-N_{0})\cup (S^{3}-M_{N-1})$ and a component $S^{\prime}$ of $N_{K-1}$ so that $h(T)\subset \text{Int}(S^{\prime})$
and so that a component $S$ of $N_{K}\cap S^{\prime}$ is interior to $h(T)$.

\textbf{Case 1. }\emph{S is a Whitehead link in $S^{\prime}$:}

The geometric index of $S$ in $S^{\prime}$ is two. Thus the geometric index of $h(T)$ in $S^{\prime}$ is non-zero and the geometric index of $S$ in $h(T)$ is non-zero. So the geometric index of $h(T)$ in $S^{\prime}$ must be one or two. If it is one, then there is a further homeomorphism of $h^{\prime}$ of $S^{3}$ that is fixed on $X$ and the complement of a small neighborhood of $S^{\prime}$ that takes $h(T)$ to $S^{\prime}$. If the geometric index of $h(T)$ in $S^{\prime}$ is two, then the geometric index of $S$ in $h(T)$ is one and there is a further homeomorphism of $h^{\prime}$ of $S^{3}$ that is fixed on $X$ and the complement of $S^{\prime}$ that takes $h(T)$ to $S$.

\textbf{Case 2. }\emph{S is one component of a Bing link in $S^{\prime}$}:

Let $R$ be the other component of the Bing link. If $h(T)$ contains $S\cup R$, the geometric index of $S\cup R$ in $S^{\prime}$ is two, and so the geometric index of $S\cup R$ in $T$ and the geometric index of $h(T)$ in $S^{\prime}$ must be nonzero and $\leq 2$. The geometric index of each of $R$ and $S$ in $h(T)$  must then be zero since the geometric index of $R$ and $S$ in $S^{\prime}$ is zero. So the geometric index of $S\cup R$ in $h(T)$ is even and must be two, and the geometric index of $h(T)$ in $S^{\prime}$ must be one. Then there is a further homeomorphism of $h^{\prime}$ of $S^{3}$ that is fixed on $X$ and the complement of a small neighborhood of $S^{\prime}$ that takes $h(T)$ to $S^{\prime}$.

The other case where $h(T)$ contains only $S$ but not $R$ follows from the fact that there is a homeomorphism of $S^{3}$ to itself that takes $T-\text{Int}(S\cup R)$ to itself and takes 
$\partial S $ to $\partial T$. This homeomorphism follows from the  fact that 
$S\cup R\cup (S^{3}-\text{Int}(T))$ are Borromean Rings.

\end{proof}

The next result shows that once stages of defining sequences match up, all further stages can be made to match up.

\begin{cor}\label{matchWsequences}If $x$ is a component of the generalized BW compactum with 
two defining sequences $(M_{i})$  and  $(N_{j} )$ and with $M_0$ = $N_0$, then the Whitehead sequence of $x$ with respect to $(M_{i})$ is the same as the Whitehead sequence of $x$ with respect to $(N_{j})$.
\end{cor}
\begin{proof}
It suffices to show that there is a self homeomorphism of $M_0$ = $N_0$ that is fixed on  $\partial M_0$ = $ \partial N_0$ and on $X$ that takes $M_i$ onto $N_i$ for any specified finite number of stages. 

Suppose that such a homeomorphism $h_n$ exists that matches the components up through $n$ stages.  Let $T$ be a component of $N_n$.  Let $M$ equal $h_n(M_{n+1}) \cap T$ and $N$ equal $N_{n+1} \cap T$.  By Lemma \ref{remove curves} we may assume that the boundaries of $M$ and $N$ are disjoint.  

A geometric index argument shows that $M$ and $N$ both have the same number of components. 
If $M$ and $N$ both have one component, a geometric index argument similar to Case 1
of Lemma \ref{match stage} shows $\partial N$ is parallel to $\partial T$ or $\partial M$.  But the geometric index of $N$ in $T$ is 2 so $\partial M$ and $\partial N$ are parallel and the boundaries can be matched up with a homeomorphism of $T$ taking  $\partial M$ to $\partial N$ fixed on $X$ and $\partial T$.  The same argument works if  $N$ lies in $M$.  

Suppose now that $M$ and $N$ both have two components.  Then one component of $M$ contains or is contained in one component of $N$ and the other component of $M$ contains or is contained in the other component of $N$.  A geometric index argument similar to Case 2 of Lemma \ref{match stage}  can be used to show that $\partial M$ and $\partial N$ are parallel and as before we can get a homeomorphism fixed on $X$ and $\partial T$ taking $M$ to $N$.

Repeating this argument in each component of $N_n$ gives the homeomorphism $h_{n+1}$.
\end{proof}

 The previous  lemmas and corollary were used in \cite{GaReWrZe11} to show that if a standard BW Cantor set $X$ in $S^{3}$ has two defining sequences $(M_{i})$ and $(N_{j})$, then past some finite stage $N$, the BW pattern for $X$ from  $(M_{i})$ must be identical to the BW pattern for $X$ from $(N_{j})$. That is, the BW patterns for the two defining sequences can differ only in a finite number of Whitehead constructions. This is a consequence of the fact that the lemmas above, together with the characterization of when a Cantor set is obtained, show that if the BW patterns do not match up at some same finite stage, then the BW patterns are repeating which contradicts the fact that a Cantor set is obtained.
 
 The situation is not quite so simple for generalized BW constructions. The pattern illustrated in Figure \ref{mixing} can occur. If $A$, $B$, $C$, and $D$ represent four generalized BW compacta, they can be assembled into a single BW compactum in either of the two ways (and in other ways) indicated in this Figure. On the left, the constructions for $A$, $B$, $C$, and $D$ are performed after doing two initial Bing constructions. On the right, the construction for $A$ is done in one of the two tori of an initial Bing construction. The construction for $B$ is done in one of the two tori after another Bing construction in the component not corresponding to $A$. The constructions for $C$ and $D$ are done at the next stage.
 
\begin{figure}[htb]
\begin{minipage}[b]{0.45\linewidth}
\begin{center}
\includegraphics[width=.85\textwidth]{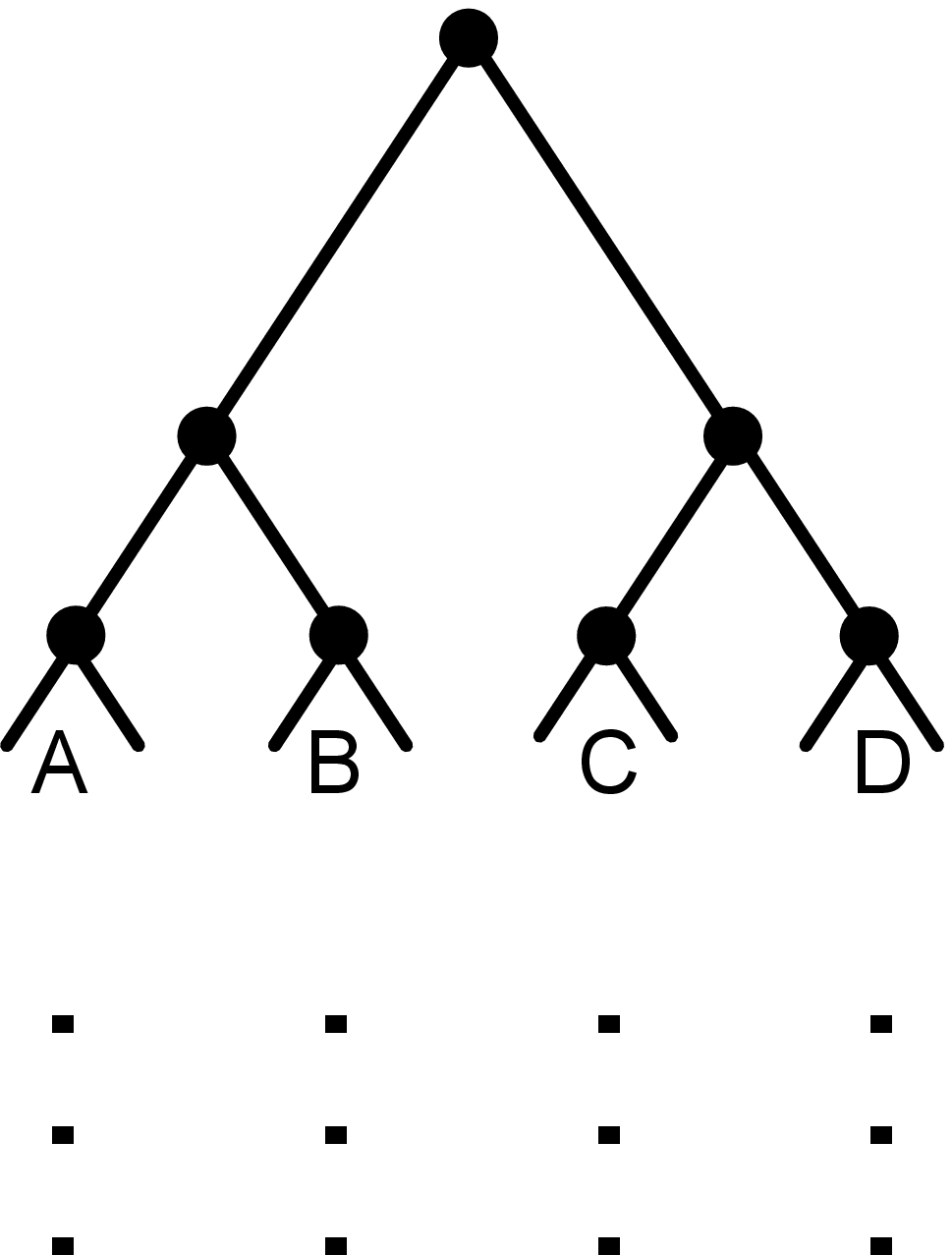}
\end{center}
\end{minipage}
\hfill
\begin{minipage}[b]{0.45\linewidth}
\begin{center}
\includegraphics[width=.75\textwidth]{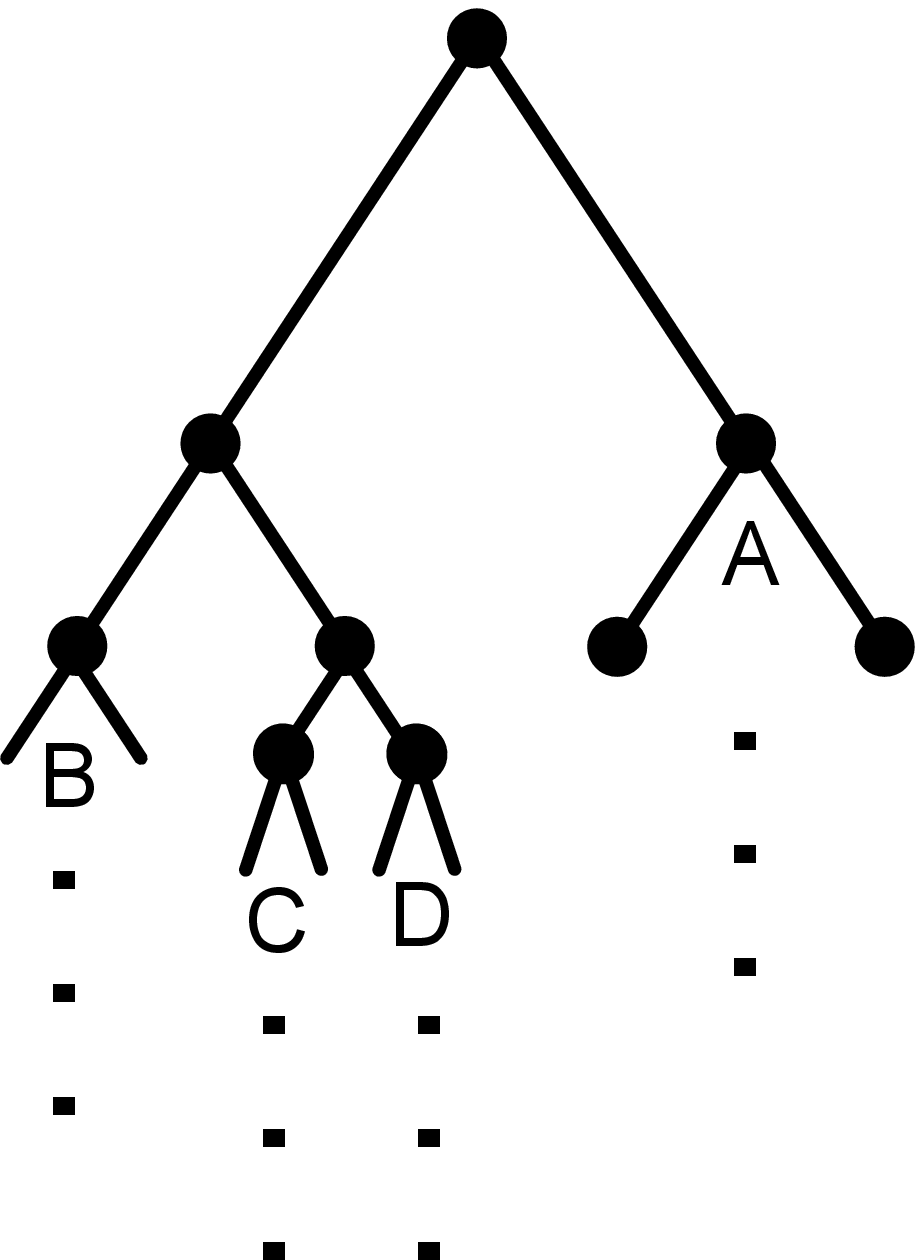}
\end{center}
\end{minipage}
\caption{Equivalent Generalized Constructions}
\label{mixing}
\end{figure}

We can however analyze the following situation that is needed in proving rigidity (the case $X=Y$ of the following).
\begin{thm}\label{sametails}
Assume that $X$ and $\,Y$ are generalized BW {compacta}  
with  defining  sequences $(M_{i})$  and  $(N_{j})$. Let $c$ be a component of $X$. Assume there is a self homeomorphism $h$ of $S^3$ that takes $X$ to $Y$ and takes $c$ to a component $d$ of $Y$. Let $s(c)$ be the Whitehead sequence associated with
$c$ from $(M_{i})$ and let $s(d)=s(h(c))$  be the Whitehead sequence associated with $d$ 
from $(N_{j})$ as in  Definition \ref{Wsequence}. Then these two sequences have the same tails. That is, there is a stage $n$ in the first sequence and a stage $m$ in the second sequences such that for all $k$, $s(c)_{n+k}=s(d)_{m+k}$.
\end{thm}
\begin{proof}
By Corollary \ref{matchWsequences}, there is a homeomorphism of $S^{3}$, fixed on $Y$ that takes a component of some stage in the defining sequence $(h(M_i))$ for $h(c)$ to a component of some stage in the defining sequence $(N_j)$ for $d$. By viewing these stages as the initial stages of generalized BW constructions, and by applying Corollary \ref{matchWsequences}, one sees that the Whitehead sequence $s(h(c))$ with respect to $(h(M_i))$ matches up with the Whitehead sequence $s(d)$ with respect to $(N_j)$ past this stage. But the Whitehead sequence for $s(h(c))$ with respect to $(h(M_i))$ is the same as  the Whitehead sequence $s(c)$ with respect to $(M_i)$. The assertion follows.
\end{proof}

\section{Main Results}
\label{mainresultssection}

{\begin{const}\label{main construction}
{For every increasing sequence $s=(s_{1}.s_{2},s_{3},\ldots)$ of positive integers, we describe a construction of a BW compactum $C(s)$. 
The compactum $C(s)$  will be $BW(z_1, z_2, \dots)$ for a specific sequence $(z_1, z_2, \dots)$ associated with the sequence $s$. A subsequence $(n_{1},n_{2}\ldots)$ of the sequence of positive integers will be described inductively. For $i$ not in this subsequence,  $z_i$ is an ordered $2^i$-tuple of zeros.  For $i=n_{j}$ in this subsequence,  $z_{n_1}$ consists of the first $2^{n_1}$ terms in $s$, $z_{n_2}$ consists of the next $2^{n_2}$ terms in $s$, and so forth.  As before, let $w_i = \max(z_i)$ and $\sigma_i = w_1 + w_2 + \cdots w_i$.  To make sure that we can choose $C(s)$ to be a Cantor set, we just choose the $n_i$ appropriately as in Theorem \ref{Whiteheadlimit}.  Choose $n_1 = 1$, and for $i > 1$ choose $n_{i+1} = 2^{\sigma_i} +n_i$.  
See Figure \ref{construction}.}

\begin{figure}[ht]
\begin{center}
\includegraphics[width=.9\textwidth]{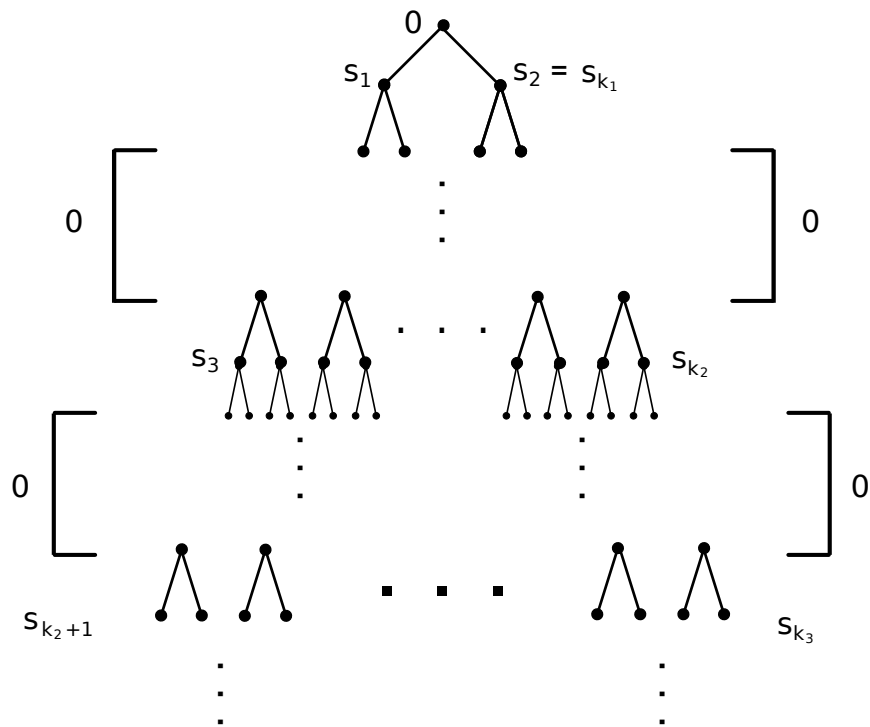}
\end{center}
\caption{Stages in Main Construction}
\label{construction}
\end{figure}

\end{const}

\begin{thm}\label{main theorem}
For every increasing sequence $s=(s_{1}.s_{2},s_{3},\ldots)$, the BW compactum $C(s)$  above can be constructed so as to be a rigid generalized BW Cantor set $C(s)\subset S^{3}$ with simply connected complement.
\end{thm}
\begin{proof}
The complement of the compactum $C(s)$ is simply connected by Lemma \ref{BW properties}.
{ It  is an easy matter to check that $\displaystyle \sum_{i=1}^\infty {1 \over 2^{\sigma_i} }= \infty$ and so it is possible to construct $C(s)$ to be a Cantor set by Theorem \ref{Whiteheadlimit}.}
 It remains to be shown that $C(s)$ is rigidly embedded. 

Let $h$ be a self homeomorphism of $S^{3}$ that takes $C(s)$ to itself. If $h$ restricted to $C(s)$ is not the identity, then there are points $x\neq y$ in $C(s)$ with $h(x)=y$. The construction guarantees that the Whitehead sequences $s(x)$ and $s(y)$ do not have the same tail, whereas 
Theorem \ref{sametails} requires that they do. This contradiction establishes the result.
\end{proof}

\begin{lem}\label{distinct rigid}If two sequences $s$ and $t$ with only finitely many terms in common are used to construct Cantor sets $C(s)$ and $C(t)$ as in Theorem \ref{main theorem}, then $C(s)$ and $C(t)$ are inequivalently embedded.
\end{lem}
\begin{proof} Assume there is a homeomorphism of $S^{3}$ taking $C(s)$ to $C(t)$. The construction guarantees that no point in $C(s)$ has Whitehead sequence with the same tail as a point in $C(t)$, contradicting Lemma \ref{sametails}.
\end{proof}

\begin{thm}\label{maintheorem2}
There are uncountably many inequivalent BW Cantor sets in $S^{3}$ with simply connected complement.
\end{thm}
\begin{proof} This follows from Lemma \ref{distinct rigid} and the fact that there are uncountably many increasing sequences of positive integers, any two of which have only finitely many terms in common. One easy way to see this is to sequentially label the vertices of an infinite binary tree
with the positive integers, and choose the sequences corresponding to descending {rays} through the tree.
\end{proof}

\section{Application to $3$-manifolds}
\label{3manifoldsection}

\begin{thm} \label{maintheorem3} For every generalized BW Cantor set $C(s)$ as in Theorem \ref{main theorem}, there is an open 3-manifold $M(s)$ with end set $C(s)$ with the following properties:
\begin{itemize}
\item $M(s)$ is simply connected,
\item The Freudenthal (endpoint) compactification of $M(s)$ is $S^{3}$,
\item Any self homeomorphism $h$ of $M(s)$ extends to a unique homeomorphism $\overline{h}$ of $S^{3}$, 
\item The homeomorphism $\overline{h}$ restricted to the ends is the identity, and
\item $M(s)$ has genus one at infinity.
\end{itemize}
 \end{thm}
\begin{proof}Given $C(s)$ as in Theorem \ref{main theorem}, let $M(s)$ be the complement of $C(s)$ in $S^{3}$. Theorem \ref{main theorem} shows that $M(s)$ is simply connected. The definition of $M(s)$ shows that the end set is $C(s)$ and that the Freudenthal (endpoint) compactification of $M(s)$ is $S^{3}$. Standard results about endpoint compactifications show that
any self homeomorphism $h$ of $M(s)$ extends to a unique homeomorphism $\overline{h}$ of $S^{3}$. The extension homeomorphism $\overline{h}$ must be the identity when restricted to the ends since $C(s)$ is rigidly embedded. The genus at infinity of $M(s)$ is the same as the genus of $C(s)$ which is one since $C(s)$ has a defining sequence with genus one components, and since no point of $C(s)$ is tamely embedded.
\end{proof}
\begin{cor} There are uncountably many distinct 3-manifolds having the properties listed in Theorem \ref{maintheorem3}.
\end{cor}
\begin{proof} This follows directly from Theorem \ref{maintheorem2} since a homeomorphism between $M(s)$ and $M(t)$ implies equivalence of embeddings of $C(s)$ and $C(t)$.
\end{proof}

\section{Questions}
\label{question section}
The following questions arise from a consideration of the results in this paper.

\begin{que}
For the  simply connected 3-manifolds constructed in this paper, each with end set a Cantor set, is every self homeomorphism isotopic either to an involution fixing the ends or to the identity?
\end{que}
\begin{que}
Does there exist a simply connected 3-manifold with only one end so that every self homeomorphism is isotopic either to an involution fixing the ends or  to the identity?
\end{que}
\begin{que}
Can the techniques of this paper be used to show there exist rigid embeddings of compacta in $S^3$ with simply connected complement?
\end{que}
\begin{que}
What is the group of isotopy classes of self homeomorphisms of the Whitehead 3-manifold? The expected answer would be  $\mathbb Z \times \mathbb Z_2$, where the first factor comes from ratcheting \cite{Wr92} and the second factor from flipping.
\end{que}


\section{Acknowledgments}
The first author was supported in part by the National Science Foundation grant DMS0852030. The first and second authors were supported in part by the Slovenian Research Agency grant BI-US/11-12-023. 
The first and third authors were supported in part by the National Science Foundation grant DMS1005906. 
The second author was supported in part by the Slovenian Research Agency grants  P1-0292-0101 and J1-2057-0101.  The authors would also like to thank the referees for helpful comments and suggestions.

\bibliographystyle{amsalpha}
\bibliography{BW}

\end{document}